\documentclass[11pt]{article}    
\usepackage{amsmath,amssymb,amsthm,amsfonts}

\title{A Characterization of Saturated Designs for Factorial Experiments}

\author{Roberto Fontana \\ Department of Mathematical Sciences, Politecnico di Torino
\and Fabio Rapallo \\ Department DISIT, Universit\`a del Piemonte
Orientale \and Maria Piera Rogantin \\ Department of Mathematics,
Universit\`a di Genova }

\usepackage{fancybox}

\usepackage{multirow}

\usepackage{color}
\usepackage{graphicx}

\date{}
\newtheorem{theorem}{Theorem}[section]
\newtheorem{definition}{Definition}[section]
\newtheorem{proposition}{Proposition}[section]

\newtheorem{remark}{Remark}[section]
\newtheorem{example}{Example}[section]
\newcommand{\design}{{\mathcal D}}
\newcommand{\fraction}{{\mathcal F}}

\begin{document}

\maketitle

\begin{abstract}
In this paper we study saturated fractions of factorial designs under the perspective of Algebraic Statistics. We define a criterion to check whether a fraction is saturated or not with respect to a given model. The proposed criterion is based purely on  combinatorial objects. Our technique is particularly useful when several fractions are needed. We also show how to generate random saturated fractions with given projections, by applying the theory of Markov bases for contingency tables.

\end{abstract}

\section{Introduction}

The search for minimal designs to estimate linear models is an active research area
in the design of experiments. Given a model, saturated fractions
have a minimum number of points to estimate all the parameters of a given model.
As a consequence, all information is used to estimate the parameters and there is no degree of freedom to
estimate the error term. Nevertheless, saturated fractions are of
common use in sciences and engineering, and they become
particularly useful for highly expensive experiments, or when time
limitations impose the choice of the minimum possible number of
design points. For general reference in the design of experiments,
the reader can refer to \cite{raktoeetal:81} and \cite{bailey:08},
where the issue of saturated fractions is discussed.


In this paper we characterize saturated fractions of a factorial design
in terms of the circuits of their design matrices and
we define a criterion to check whether a given fraction is saturated or not.
This avoids computating the determinant of the corresponding design matrix.

Our work falls within the framework of Algebraic statistics. The application of polynomial
algebra to the design of experiments was originally presented in
\cite{pistoneetal:01}, but with a different point of view. The
techniques used here are mainly based on the combinatorial and algebraic objects associated to the design matrix of the model, such as the circuits, the Graver basis and the Universal Gr\"obner basis. These algebraic objects are different bases of the toric ideal of the design matrix, which are a special set of polynomials originally used in Statistics for the analysis contingency tables. All these bases are used to solve enumeration problems, to make non-asymptotic
inference, and to describe the geometric structure of the
statistical models for discrete data. A recent account of this
theory can be found in \cite{drtonetal:09}.

In this paper, we benefit from the interplay between algebraic
techniques for the analysis of contingency tables and some topics
of the design of experiments. We identify
a factorial design with a binary contingency table whose entries are the
indicator function of the fraction, i.e., they are equal to $1$
for the fraction points and $0$ for the other points. This implies that a
fraction can also be considered as a subset of cells of the table.
Some recent results in this direction
can be found in \cite{aoki|takemura:10}. The connections between
experimental designs and contingency tables have also been
explored in \cite{fontanaetal:12}, but were limited to the
investigation of enumerative problems in the special cases of
contingency tables arising from the Sudoku problems.

The basic idea underlying our theory is as follows. A fraction is saturated if a null linear combination of its points with non-zero coefficients exists. This vector of coefficients belongs to the kernel of the transpose of the design matrix. IN algebraic language this means generating bases of the toric ideal associated to the design matrix. Thus, each null linear combination with integer coefficients translates into a binomial of the toric ideal of the design matrix. In this paper we prove that if we consider a special basis of the toric ideal, namely the circuit basis, this is also a sufficient condition. Our approach based on circuits avoids the computation of the determinant of the design matrix, and therefore it avoids possible numerical problems. It is particularly useful when we need to find several saturated fractions for the same model as in the case of algorithms for optimal design generation. Furthermore, the circuits can be computed once and for all from the design matrix of the full factorial model and do not depend on the fraction.

The paper is organized along these lines. In Section \ref{sat-frac} we
set some notations and we state the problem. In Section
\ref{Alg-Stat} we provide the basic algebraic framework to be used
in the paper. In Section \ref{main-res} we prove the main result,
showing that the absence of circuits is a necessary and sufficient
condition for obtaining a saturated fraction. In Section
\ref{sec:comp}, for unimodular design matrices, we report results which highlight the relationship between the different bases of the relevant toric ideal, and we show that for several models the design matrix is unimodular. In Section \ref{sect:ex} we provide some examples to demonstrate the practical applicability of our theory. In Section \ref{sect:gen} we show how to generate a sample of saturated fractions, by applying the theory of Markov bases when we add constraints to the projections. Finally, in Section \ref{fut-dir} we suggest some future research directions
stimulated by the theory presented here.

\section{Saturated designs} \label{sat-frac}

Let $\design$ be a full factorial design with $d$ factors, $A_1,
\ldots, A_d$ with $s_1, \ldots , s_d$ levels respectively,
$\design= \{0,\dots,s_1-1\}
\times \cdots \times  \{0,\dots,s_d-1\}$. We consider a linear model
on $\design$:
\begin{equation*}
Y = X \beta + \varepsilon \, ,
\end{equation*}
where $Y$ is the response variable, $X$ is the {\em design
matrix}, $\beta$ is the vector of model parameters, and
$\varepsilon$ is a vector of random variables that
represent the error terms. We denote by $p$ the number of
estimable parameters.

For instance, in a two-factor design with the simple effect model,
we have $p=s_1+s_2-1$ and a possible design matrix is:
\begin{equation}   \label{mat-repr}
X= \left( m_0 \ | \ a_0 \ | \ \ldots \ | \ a_{s_1-2} \ | \ b_0 \ | \
\ldots \ | \ b_{s_2-2}  \right) \, ,
\end{equation}
where $m_0$ is a column vector of $1$'s, $a_0, \ldots, a_{s_1-2}$
are the indicator vectors of the first $(s_1-1)$ levels of the
factor $A_1$, and $b_0, \ldots, b_{s_2-2}$ are the indicator vectors
of the first $(s_2-1)$ levels of the factor $A_2$.

A subset $\fraction$, or fraction, of a full design $\design$,
with minimal cardinality $\#\fraction=p$, that allows us to
estimate the model parameters, is a saturated fraction or \emph{saturated design}. By
definition, the design matrix $X_\fraction$ of a saturated design
is a non-singular matrix with dimensions $p \times p$.

\begin{example} \label{first-ex}
Let us consider the $2^4$ design and the model with simple effects and 2-way interactions. This example is at the same time not trivial and easy to handle, so that we will use it as the running example in this paper. The design matrix $X$ of the full design has rank equal to $11$ and is reported in Figure \ref{ex_fraction}.
\begin{figure}
\begin{footnotesize}
\begin{equation*}
X=\bordermatrix{&1 &a_0 &b_0 & c_0& d_0& a_0b_0 & a_0c_0 & a_0d_0 & b_0c_0 & b_0d_0 & c_0d_0 \cr
                (0,0,0,0)&1 & 1 & 1 & 1 & 1 & 1 & 1 & 1 & 1 & 1 & 1 \cr
                (0,0,0,1)&1 & 1 & 1 & 1 & 0 & 1 & 1 & 0 & 1 & 0 & 0 \cr
                (0,0,1,0)&1 & 1 & 1 & 0 & 1 & 1 & 0 & 1 & 0 & 1 & 0 \cr
                (0,0,1,1)&1 & 1 & 1 & 0 & 0 & 1 & 0 & 0 & 0 & 0 & 0 \cr
                (0,1,0,0)&1 & 1 & 0 & 1 & 1 & 0 & 1 & 1 & 0 & 0 & 1 \cr
                (0,1,0,1)&1 & 1 & 0 & 1 & 0 & 0 & 1 & 0 & 0 & 0 & 0 \cr
                (0,1,1,0)&1 & 1 & 0 & 0 & 1 & 0 & 0 & 1 & 0 & 0 & 0 \cr
                (0,1,1,1)&1 & 1 & 0 & 0 & 0 & 0 & 0 & 0 & 0 & 0 & 0 \cr
                (1,0,0,0)&1 & 0 & 1 & 1 & 1 & 0 & 0 & 0 & 1 & 1 & 1 \cr
                (1,0,0,1)&1 & 0 & 1 & 1 & 0 & 0 & 0 & 0 & 1 & 0 & 0 \cr
                (1,0,1,0)&1 & 0 & 1 & 0 & 1 & 0 & 0 & 0 & 0 & 1 & 0 \cr
                (1,0,1,1)&1 & 0 & 1 & 0 & 0 & 0 & 0 & 0 & 0 & 0 & 0 \cr
                (1,1,0,0)&1 & 0 & 0 & 1 & 1 & 0 & 0 & 0 & 0 & 0 & 1 \cr
                (1,1,0,1)&1 & 0 & 0 & 1 & 0 & 0 & 0 & 0 & 0 & 0 & 0 \cr
                (1,1,1,0)&1 & 0 & 0 & 0 & 1 & 0 & 0 & 0 & 0 & 0 & 0 \cr
                (1,1,1,1)&1 & 0 & 0 & 0 & 0 & 0 & 0 & 0 & 0 & 0 & 0 \cr }
\end{equation*}
\end{footnotesize}
 \caption{The design matrix $X$ of the model in Example \ref{first-ex}.} \label{ex_fraction}
\end{figure}
As the matrix $X$ has rank $11$, we search for fractions with $11$ points. For instance the fraction
\begin{equation*}
\begin{split}
\fraction_1 = \{ (0,0,0,0),
(0,0,0,1),(0,0,1,1),(0,1,0,1),(0,1,1,0),(1,0,0,1),\\ (1,0,1,0),
(1,1,0,0),(1,1,0,1),(1,1,1,0),(1,1,1,1) \}
\end{split}
\end{equation*}
is saturated while
\begin{equation*}
\begin{split}
\fraction_2 = \{(0,0,0,0),
(0,0,0,1),(0,0,1,1),(0,1,0,0),(0,1,1,0),(1,0,0,1),\\ (1,0,1,0),
(1,1,0,0),(1,1,0,1),(1,1,1,0),(1,1,1,1) \}
\end{split}
\end{equation*}
is not. A direct computation shows that there are $\binom{16}{11}=4,368$
fractions with $11$ points: among them $3,008$ are saturated, and the
remaining $1,360$ are not.
\end{example}

\section{Designs and contingency tables with Algebraic Statistics}
\label{Alg-Stat}

As mentioned in the Introduction, we identify a factorial design
with a contingency table whose entries are the indicator function
of the fraction. In the previous example, $\fraction_1$ is also a
$2^4$ table with $1$ in the cells identified by the point coordinates of $\fraction_1$ and $0$
otherwise as in Table \ref{tabella}. To avoid misunderstandings, $\fraction$ denotes the fraction, while $N(\fraction)$ denotes the corresponding binary table. In Table \ref{tabella} the contingency table
representations of the fractions $\fraction_1$ and $\fraction_2$
from Example \ref{first-ex} are displayed.

\begin{table}
\begin{center}
\begin{tabular}{ccc}
$N(\fraction_1)$ & \qquad &
\begin{tabular}{c|c|cccc} \hline
& & \multicolumn{2}{c}{$A_1=0$} & \multicolumn{2}{c}{$A_1=1$} \\ \hline
 & & $A_2=0$ & $A_2=1$ & $A_2=0$ & $A_2=1$ \\ \hline
 \multirow{2}{*}{$A_3=0$} & $A_4=0$ & 1 & 1 & 0 & 1 \\
  & $A_4=1$ & 1 & 0 & 1 & 1 \\
 \multirow{2}{*}{$A_3=1$} & $A_4=0$ & 0 & 1 & 1 & 1 \\
  & $A_4=1$ & 1 & 0 & 0 & 1 \\ \hline
\end{tabular}
\end{tabular}
\qquad
\begin{tabular}{ccc}
$N(\fraction_2)$ & \qquad &
\begin{tabular}{c|c|cccc} \hline
 & & \multicolumn{2}{c}{$A_1=0$} & \multicolumn{2}{c}{$A_1=1$} \\ \hline
 & & $A_2=0$ & $A_2=1$ & $A_2=0$ & $A_2=1$ \\ \hline
 \multirow{2}{*}{$A_3=0$} & $A_4=0$ & 1 & 0 & 0 & 1 \\
 & $A_4=1$ & 1 & 1 & 1 & 1 \\
 \multirow{2}{*}{$A_3=1$} & $A_4=0$ & 0 & 1 & 1 & 1 \\
 & $A_4=1$ & 1 & 0 & 0 & 1 \\ \hline
\end{tabular}
\end{tabular}
\caption{$4$-way contingency tables $N(\fraction_1)$ and $N(\fraction_2)$ from Example \ref{first-ex}.}
\label{tabella}
\end{center}
\end{table}

Such an identification leads us to use the Algebraic Statistics
tools from both contingency tables and Design of Experiments
theories.

In order to give a complete account of our theory and to present
our algorithm with full details, some
definitions and a few basic facts concerning the algebraic
representation of contingency tables and the combinatorial
properties of some statistical models are reported. For the Algebraic
Statistics notions, the reader can refer to \cite{drtonetal:09} or
\cite{pistoneetal:01}. For details on Algebraic notions, see
\cite{coxetal:92} and \cite{kreuzer|robbiano:00},

Given a contingency table with $K$ cells, we consider the
polynomial ring ${\mathbb R}[x] = {\mathbb R}[x_1, \ldots, x_K]$
of all polynomials with indeterminates $x_1, \ldots, x_K$ and real
coefficients. The relevant polynomial ring has one indeterminate
for each cell of the table (or equivalently, for each point of the
design). An ideal ${\mathcal I}$ in ${\mathbb R}[x]$ is a subset
of ${\mathbb R}[x]$ such that $f+g \in {\mathcal I}$ for all $f,g
\in {\mathcal I}$ and $fg \in {\mathcal I}$ for all $f \in
{\mathcal I}$ and for all $g \in {\mathbb R}[x]$.

In Algebraic Statistics, a class of ideals is of special interest.
Given a $s \times K$ non-negative matrix $A$ with integer entries,
the toric ideal defined by $A$ is the binomial ideal
\begin{equation*}
{\mathcal I}_A = \left\{x^a - x^b \ : \ Aa = Ab \right\}
\end{equation*}
where the monomials $x^a$ are written in vector notation $x^a= x_1^{a_1} \cdots x_K^{a_K}$.

Thanks to the Hilbert's basis theorem, every ideal has a finite
basis $\{f_1, \ldots, f_n\}$: for all $f \in {\mathcal I}$ there
are polynomials $g_1, \ldots, g_n \in {\mathbb R}[x]$ such that
$f=g_1f_1 + \cdots + g_nf_n$. The computation of a system of
generators of an ideal is a non trivial task in Computer Algebra.
An actual way to do that is to compute the reduced Gr\"obner basis
of the ideal. The computation of a Gr\"obner basis depends on the
term-order chosen in the polynomial ring ${\mathbb R}[x]$, but for
a given term-order the reduced Gr\"obner basis is unique and can
be computed through symbolic software.

Among all term-orders, the elimination term-order for a given
indeterminate, say $x_K$, leads to the Gr\"obner basis of the
projection ${\rm Elim}(x_K ; {\mathcal I}) := {\mathcal I} \cap
{\mathbb R}[x_1, \ldots, x_{K-1}]$, just taking the Gr\"obner
basis of ${\mathcal I}$ and removing the polynomials involving
$x_K$. Applying this fact iteratively, Theorem 4 in
\cite{rapallo:06} shows that the statistical counterpart of
elimination of indeterminates is the definition of a statistical
model for incomplete tables.

As there are finitely many term-orders, there are finitely many
Gr\"obner bases.

\begin{definition}
Let ${\mathcal I}_A$ be a toric ideal. The union of all reduced
Gr\"obner bases of ${\mathcal I}_A$ is called the {\em Universal
Gr\"obner basis} ${\mathcal U}_A$ of ${\mathcal I}_A$.
\end{definition}

The computation of the Universal Gr\"obner basis is unfeasible for
most ideals, but fortunately there are special algorithms for doing
that in the case of toric ideals.

There are now several computer systems for handling multivariate
polynomials, see for instance \cite{cocoa} and \cite{singular:12},
and all of them compute Gr\"obner bases. Recently, also the {\tt
R} package {\tt mpoly} has been implemented, see \cite{mpoly:13}.
For toric ideals the fastest algorithms are implemented in {\tt
4ti2}, see \cite{4ti2}.

Together with the Universal Gr\"obner basis, there are other
combinatorial and polynomial objects derived from a nonnegative
integer matrix $A$.

\begin{definition}
A binomial $f = x^a-x^b \in {\mathcal I}_A$
is primitive if there is no binomial $g = x^c-x^d \in {\mathcal
I}_A$, with $g \ne f$, such that $c \leq a$ and $d \leq b$. The
{\em Graver basis} $Gr_A$ of $A$ is the set of all primitive
binomials in ${\mathcal I}_A$.
\end{definition}

\begin{definition}
The support of a binomial $f = x^a-x^b$ is the set of indices $i$
($i=1, \ldots , K$) such that $a(i) \ne 0$ or $b(i) \ne 0$. We
denote the  support of $f$ with ${\rm supp}(f)$.
\end{definition}

\begin{definition}
An irreducible binomial $f = x^a-x^b \in {\mathcal I}_A$ is a
circuit if there is no other binomial $g \in {\mathcal I}_A$ such
that ${\rm supp}(g) \subset {\rm supp}(f)$ and ${\rm supp}(g) \ne {\rm supp}(f)$. We
denote the set of all circuits of ${\mathcal I}_A$ with ${\mathcal
C}_A$.
\end{definition}

It is easy to see that every circuit is primitive. Moreover,
Theorem $1.1$ in \cite{ohsugi:12} states that
\begin{equation} \label{inclus}
{\mathcal C}_A \subseteq {\mathcal U}_A \subseteq Gr_A \, .
\end{equation}
We will come back on such inclusions later in Section \ref{sec:comp}.

\begin{remark}
The notion of Universal Markov basis in Statistics for contingency
tables has been already introduced in \cite{rapallo|rogantin:07} and \cite{rapallo|yoshida:10} for the analysis of bounded tables. The circuits are shown to be relevant objects for finding all the supports of the probability distributions in the closure of an exponential family for finite sample spaces, see \cite{rauhetal:11}.
\end{remark}

\section{Characterization of saturated designs} \label{main-res}

We are now ready to state the main result, exploiting the
connections between the saturated fractions and the circuits of
the design matrix. Recall that the support of a binomial
$f=x^u-x^v$ is the set of indices for which $u>0$ or $v>0$.

In the algebraic theory of toric ideals it is common to use the transpose of the design matrix $X$ in place of the design matrix $X$. In order to simplify the notation, we denote by $A=X^t$ the transpose of the design
matrix, and with a slight abuse of notation we call it {\it design matrix}. As a consequence, given ${\mathcal F}=\left\{i_1, \ldots, i_p \right\}$, $A_{\mathcal F}$ is the submatrix of $A$ obtained by selecting the columns of $A$ according to $\fraction$. Note that each column of $A$ identifies a design point, and therefore the definition of a set of column-indices is equivalent to the definition of the fraction with the corresponding design points.

\begin{theorem} \label{theo1}
Let $A$ be a full-rank design matrix with dimensions $p \times K$
and let ${\mathcal C}_A=\{f_1 , \ldots, f_L\}$ be the set of its
circuits. Given a set ${\mathcal F}$ of $p$ column-indices of $A$, the
submatrix $A_{\mathcal F}$ is non-singular if and only if $
{\mathcal F}$ does not contains any of the supports ${\rm
supp}(f_1), \ldots , {\rm supp}(f_L)$.
\end{theorem}

\begin{proof}
We prove that $A_{\mathcal F}$ is singular if and only if there is
a binomial $f$ in ${\mathcal C}_A$ with ${\rm supp}(f) \subseteq
{\mathcal F}$.

``$\Rightarrow$'': If $A_{\mathcal F}$ is singular, then there is
a null linear combination of its columns. As the entries of $A_{\mathcal
F}$ are nonnegative integers, the linear combination has
coefficients in ${\mathbb Q}$, and hence there is a linear
combination with coefficients in ${\mathbb Z}$.

Therefore, there exists a non-zero vector $m \in {\mathbb Z}^p$
with $A_{\mathcal F}m=0$. Using the positive part of $m$, $m^+=\max\{m,0\}$, and the negative part $m^-=-\min(m,0)$, we decompose $m=m^+-m^-$, and the binomial $x^{m^+}-x^{m^-}$ belongs to
${\mathcal I}_{A_{\mathcal F}}$ the toric ideal associated to $A_\fraction$.

As ${\mathcal I}_{A_{\mathcal F}}$ can be computed from ${\mathcal I}_A$ with the elimination algorithm as described in Section \ref{Alg-Stat}, and in particular ${\mathcal I}_{A_{\fraction}} = {\rm Elim}(x_i \ : \ i = 1, \ldots, K, i \notin {\mathcal F} ;  {\mathcal I}_A)$, ${\mathcal I}_{A_{\mathcal F}}$ is non-empty and there is a
binomial in the Gr\"obner basis ${\mathcal G}_{A,\tau}$, where
$\tau$ is the elimination term order for the indeterminates $(x_i
\ : \ i=1, \ldots, K, i \notin {\mathcal F})$. By definition of Universal Gr\"obner
basis, this implies that there is a binomial $g$ in ${\mathcal
U}_A$ with support in ${\mathcal F}$.

If the binomial $g$ is a circuit, the proof if complete. If not,
there do exist a circuit $h \in {\mathcal C}_A$ with ${\rm supp}(h)
\subset {\rm supp}(g)$ and it is enough to choose such binomial $h$.

``$\Leftarrow$'': Suppose that there is a circuit $f \in {\mathcal
C}_A$ with support in ${\mathcal F}$, i.e., $x^{m^+}-x^{m^-}$ is
in ${\mathcal I}_{A}$. Hence, $Am=0$ and projecting onto the
subspace ${\mathbb R}[x_i \ : \ i \in {\mathcal F}]$ we obtain
$A_{\mathcal F}m=0$, i.e., $A_{\mathcal F}$ is singular.
\end{proof}

Theorem \ref{theo1} replaces a linear algebra condition (namely the non-singularity of the design matrix) with a combinatorial property for checking weather a fraction is saturated or not. This technique highlight an interesting combinatorial property of the saturated fractions, with a theoretical interest. It may be of limited practical interest when analyzing a single fraction, but it becomes useful when we need to study all the saturated fractions of a given factorial design. Note that the set of the circuits ${\mathcal C}_A$ of a design matrix can be computed once for all fractions.

We are now ready to analyze and discuss some examples. To ease the
presentation, the binomials are defined through their exponents,
so that a nonnegative vector $a-b$ is used in place of
$x^a-x^b$. For instance the binomial $x_1x_3^2-x_2x_7x_8$ in ${\mathbb R}[x_1, \ldots, x_8]$ is written as $(1,-1,2,0,0,0,-1,-1)$. This notation is the standard one in {\tt 4ti2}, the
software we use for our computations.

\begin{example} \label{remake}
We consider again the $2^4$ design and the model with simple effects and
2-way interactions, already discussed in Example
\ref{first-ex}. In less than one second, {\tt 4ti2} produces the
list of all $140$ circuits of the design matrix. Labeling the
design points lexicographically, the $140$ circuits can be divided
into three classes, up to permutations of factors or levels:
\begin{itemize}
\item 20 circuits of the form
\begin{equation} \label{firstkind}
f_1=(0,0,0,0,1,-1,-1,1,-1,1,1,-1,0,0,0,0)
\end{equation}
with $\max|f_1| = 1$;
\item 40 circuits of the form
\begin{equation} \label{secondkind}
f_2=(1,-2,0,1,0,1,-1,0,0,1,-1,0,-1,0,2,-1)
\end{equation}
with $\max|f_2| = 2$;
\item 80 circuits of the form
\begin{equation} \label{thirdkind}
f_3=(1,0,-2,1,0,-1,1,0,-2,1,3,-2,1,0,-2,1)
\end{equation}
with $\max|f_3| = 3$.
\end{itemize}

The supports of these circuits are displayed in Figure
\ref{tab_supporti}. Note that the support of $f_2$ is contained in the
fraction ${\mathcal F}_2$ of Example \ref{first-ex}, making that fraction non-saturated.
\begin{table}
\begin{center}
\begin{tabular}{c}
\begin{tabular}{c|c|cccc} \hline
& & \multicolumn{2}{c}{$A_1=0$} & \multicolumn{2}{c}{$A_1=1$} \\
\hline
 & & $A_2=0$ & $A_2=1$ & $A_2=0$ & $A_2=1$ \\ \hline
 \multirow{2}{*}{$A_3=0$} & $A_4=0$ &  & $\bullet$ & $\bullet$ &  \\
  & $A_4=1$ &   & $\bullet$ & $\bullet$ &  \\
 \multirow{2}{*}{$A_3=1$} & $A_4=0$ &  & $\bullet$ & $\bullet$ &  \\
  & $A_4=1$ &  & $\bullet$ & $\bullet$ &  \\ \hline
\end{tabular} \\
\begin{tabular}{c|c|cccc} \hline
& & \multicolumn{2}{c}{$A_1=0$} & \multicolumn{2}{c}{$A_1=1$} \\
\hline
 & & $A_2=0$ & $A_2=1$ & $A_2=0$ & $A_2=1$ \\ \hline
 \multirow{2}{*}{$A_3=0$} & $A_4=0$ & $\bullet$ &  &  & $\bullet$ \\
 & $A_4=1$ & $\bullet$ & $\bullet$ & $\bullet$ &  \\
 \multirow{2}{*}{$A_3=1$} & $A_4=0$ &  & $\bullet$ & $\bullet$ & $\bullet$ \\
 & $A_4=1$ & $\bullet$ &  &  & $\bullet$ \\ \hline
\end{tabular} \\
\begin{tabular}{c|c|cccc} \hline
& & \multicolumn{2}{c}{$A_1=0$} & \multicolumn{2}{c}{$A_1=1$} \\
\hline
 & & $A_2=0$ & $A_2=1$ & $A_2=0$ & $A_2=1$ \\ \hline
 \multirow{2}{*}{$A_3=0$} & $A_4=0$ & $\bullet$ &  & $\bullet$ & $\bullet$ \\
  & $A_4=1$ &  & $\bullet$ & $\bullet$ &  \\
 \multirow{2}{*}{$A_3=1$} & $A_4=0$ & $\bullet$ & $\bullet$ & $\bullet$ & $\bullet$ \\
  & $A_4=1$ & $\bullet$ & & $\bullet$ & $\bullet$ \\ \hline
\end{tabular}
\end{tabular}

\caption{The supports of the three circuits in Example
\ref{remake}. The bullet symbol denotes the cells in the support} \label{tab_supporti}
\end{center}
\end{table}
\end{example}

\begin{remark}
Theorem \ref{theo1} allows us to identify saturated designs with the
feasible solutions of an integer linear programming problem. Let
$\overline{{\mathcal C}}_A=(c_{ij}, i=1,\ldots,L, j=1,\ldots,K)$ be
the $L \times K$ matrix, whose rows contain the values of the
indicator functions of the supports of the circuits $f_1, \ldots,
f_L$, $c_{ij}=(f_{ij} \ne 0) , i=1,\ldots,L, j=1,\ldots,K$ and
$Y=(y_1,\ldots,y_K)$ be the $K$-dimensional column vector that contains the
unknown values of the indicator function of the points of ${\mathcal
F}$. The vector $Y$ must satisfy the following conditions:
\begin{eqnarray*}
\overline{{\mathcal C}}_A Y < b, \\
{1}_K^t Y =p , \\
y_i \in \{0,1\}
\end{eqnarray*}
where $b=(b_1, \ldots, b_L)$ is the column vector defined by
$b_i=\#{\rm supp}(f_i), i=1,\ldots , L$, and ${1}_K$ is the
column vector of length $K$ and whose entries are all equal to $1$.
\end{remark}

\section{Computational remarks} \label{sec:comp}

We have introduced in Section \ref{Alg-Stat} three different
objects associated with the design matrix of a model, namely
the circuits, the Graver basis and the Universal Gr\"obner basis.
In general strict inclusions among them holds, as stated in Eq. \eqref{inclus}. Therefore, it is interesting to find models for which such three sets
coincide. For instance, one may have the Graver basis theoretically determined, and in such case no further computations are needed.

The basic definition for investigating this issue is the following
one.

\begin{definition} \label{unimod:def}
A nonnegative integer matrix with rank $p$ is {\it unimodular} if all its non-zero $p \times p$ minors are equal to $\pm 1$.
A nonnegative integer matrix is {\it totally unimodular} if all its non-zero minors are equal to $\pm 1$.
\end{definition}

Of course, a totally unimodular matrix is unimodular. It follows immediately from Definition \ref{unimod:def} that the entries of a totally unimodular
matrix are $0$ and $1$, that each submatrix of a totally
unimodular matrix is again totally unimodular, and the transpose of a totally unimodular matrix is again totally unimodular. A couple of less intuitive properties are collected in the following proposition.

\begin{proposition} \label{unimod:char}
Let $A$ be a $0-1$ matrix with dimensions $p \times K$.
\begin{enumerate}
\item If for each subset $J$ of columns of $A$, there is a partition $\{J_1, J_2\}$ of $J$ such that
\begin{equation*}
\left| \sum_{j \in J_1} A_{i,j} - \sum_{j \in J_2} A_{i,j} \right| \leq 1 \ \mbox{ for } \ i=1, \ldots, p \, ,
\end{equation*}
then $A$ is totally unimodular. In particular, if each row contains at most $2$ non-zero entries, then $A$ is totally unimodular.

\item All matrices obtained by pivot operations on a totally unimodular matrix are totally unimodular.
\end{enumerate}
\end{proposition}

For the theory of totally unimodular
matrices, the reader can refer to \cite{schrijver:86}, Chapters 19
and 20. The following result can be found in \cite{sturmfels:96}, page 70.

\begin{proposition} \label{unimod-res}
If $A$ is a unimodular matrix, then
\begin{equation}
{\mathcal C}_A = {\mathcal U}_A = Gr_A \, .
\end{equation}
\end{proposition}

In view of above properties, we study the first non-trivial model,
i.e., the no-$d$-way interaction model for $d$ factors. All other
model matrices are submatrices of this matrix. For $d=2$, we reduce to
the independence model, and it is known that the design matrix of the independence model is totally
unimodular for arbitrary $s_1$ and $s_2$. Indeed, an alternative parametrization of the no-2-way interaction for two factors uses the following design matrix
\begin{equation}   \label{mat-repr-2}
\tilde X= \left( \ a_0 \ | \ \ldots \ | \ a_{s_1-1} \ | \ b_0 \ | \
\ldots \ | \ b_{s_2-2}  \right) \, ,
\end{equation}
in place of the design matrix in Eq. \eqref{mat-repr} discussed in Section \ref{sat-frac}. $\tilde X$ satisfies the hypotheses of the first item in Proposition \ref{unimod:char}.
To move to higher dimensions, we need to work with Lawrence
liftings. Given a matrix $A$, its Lawrence lifting is the block
matrix defined by
\begin{equation}
\Lambda(A) = \begin{pmatrix} A & I \\
0 & I \\
\end{pmatrix} \, ,
\end{equation}
where $I$ is the identity matrix and $0$ is a null matrix with
suitable dimensions.

As argued in \cite{sturmfels:96}, the Lawrence lifting of a totally unimodular
matrix is totally unimodular as well, and in \cite{ohsugi:12}, Section 4.1 is proved the
following fact: given a no-$d$-way interaction model for a $s_1
\times \cdots \times s_d$ design with design matrix $A$, the
no-$(d+1)$-interaction model for the $s_1 \times \cdots \times s_d
\times 2$ is the Lawrence lifting $\Lambda(A)$ of $A$. This property is derived with details in \cite{ohsugi|hibi:10}.

Combining all the facts in the discussion above, the following
theorem is proved.

\begin{theorem}
The design matrix of all models for a $s_1
\times s_2 \times 2 \times \cdots \times 2$ factorial designs is totally unimodular.
\end{theorem}

The simplest model with a non unimodular design matrix is the
no-3-way interaction model for the $3 \times 3 \times 3$
factorial design. Under the name of ``transportation problem'',
this model is fully discussed in \cite{sturmfels:96}, page 150.
However in this case the circuits and the Graver basis still
coincide. The simplest model where circuits and Graver basis
differ is the no-3-way interaction model for the $3 \times 3
\times 4$ factorial design. The results are presented in the next
section.

To compute the circuits and the Graver basis of our running
example and of the examples in the next section, we used the
commands {\tt circuits} and {\tt graver} in {\tt 4ti2},
\cite{4ti2}. All computations were carried out in less than 2
seconds on a standard PC. As usual in Computer Algebra, the
computational complexity increases quickly as the number of
indeterminates rises and it depends heavily on the degrees of
freedom of the model. For instance, the design $2^5$ under the
model with simple effects and 2-way interactions is a computationally unfeasible problem on a standard PC.

Finally, notee that the elements of the bases (Graver or
circuits) are the coefficients of the linear combination of
fraction points to have zero as result, and therefore a singular
matrix. Therefore, the elements of the bases with more than $p$
values different from zero are not of interest of our aims. In our
running examples, the circuits of the third kind like the circuit
$f_3$ in Eq. \eqref{thirdkind} can be excluded, and in the
following examples this situation comes up in other few cases.

\section{Examples} \label{sect:ex}

In this section we briefly describe the results of our
computations for some classical models.

\begin{itemize}
\item Design $2^5$; model with simple factors and 2-way and 3-way
interactions.

The saturated model has $26$ points. Circuit basis and Graver
basis are equal. They contain $3,254$ elements that can be divided
into $12$ classes, up to permutations of factors or levels. All of
them have support cardinality less than $26$ (and the maximum
cardinality is just $26$).

\item Design $2^5$; model with simple factors.

The saturated model has $6$ points. Also in this case circuit
basis and Graver basis are equal. The circuits are $353,616$
elements that can be divided into $38$ classes, up to permutations
of factors or levels. There are $259,904$ circuits with support
cardinality equal to 7, and therefore the circuits to be checked
in our algorithm are $93,712$ with support cardinality ranging
from 4 to 6.

\item Design $2\times 3 \times 4$; model with simple factors and
2-way interactions.

The saturated model has $18$ points. Circuit basis and Graver
basis are equal, and they contain $42$ elements that can be
divided into two classes. One class of $24$ elements has support
cardinality $12$, and another class of $18$ elements has support
cardinality $8$. All of them are needed for the algorithm.

\item Design $3\times 3 \times 4$; model with simple factors and
2-way interactions.

The saturated model has $24$ points. Circuit basis and Graver
basis are different. The latter has $19,722$ elements that can be
divided into $20$ classes, while the former has $17,994$ elements
with $19$ classes, all included in the Graver basis.  Both bases
contain $15,302$ elements with support cardinality smaller than or
equal to $24$. A configuration in the Graver basis which is not a
circuit is:
\begin{multline*}
(-1,1,1-1,1,0,0,-1,0,-1,-1,2,0,-1,0,1,-1,2,-1,\\0,1,-1,1,-1,1,0,-1,0,0,-2,1,1,-1,2,0,-1)
\, ;
\end{multline*}
All the elements in the Graver basis but not in the circuit basis
are permutations of this configuration.
\end{itemize}

\section{Generation of random saturated fractions} \label{sect:gen}

When a sequence of saturated fractions is needed, instead of a
single saturated fraction, an algorithm for finding such fractions
without computing the determinant of the design matrix of each
fraction may be useful.

As a first application, one can generate random fractions with $p$
points and then check whether they are saturated or not simply by
comparing the selected fractions with the circuit basis. For
instance, we have implemented that procedure for the model in our
running example, and we are able to generate $5,000$ random
saturated fractions in about 1 second on a standard PC, by
executing few lines of code. For our simulations, we used {\tt R},
see \cite{r-project}.

However, it is interesting to take a deeper look into the
connections between Markov bases and saturated fractions.
Therefore, we show how to apply the theory of Markov bases to
generate saturated fractions with given projections. Indeed, the
combinatorial objects needed to check whether a fraction is
saturated or not are essentially the same as those needed to
define a Markov Chain Monte Carlo sampler.

Remember that we have already identified a fraction with a binary
contingency table. Given a fraction $\fraction$, the corresponding
table is $N({\mathcal F})$, where $N({\mathcal
F})_{i_1,\ldots,i_d}=1$ if $(i_1,\ldots, i_d)$ is a point of
$\fraction$ and $N({\mathcal F})_{i_1,\ldots,i_d}=0$ otherwise.
Using Algebraic statistics tools, we are able to generate all
fractions with given margins, or projections, through a Markov chain algorithm
following the theory in \cite{rapallo|rogantin:07}. Moreover, we
can merge that algorithm with our theory in order to select only
the saturated fractions with given margins. Interestingly, the
Algebraic and Combinatorial objects needed to define the relevant
Markov chain for navigating the set of all fractions with given
margins and for checking whether the fraction is saturated are
very close.

Basic facts about Markov bases are reported. The reader
can refer to the book \cite{drtonetal:09} for a complete
presentation. A \emph{Markov move} $m$ is a table with integer
entries such that $N({\mathcal F})$, $N({\mathcal F})+m$ and
$N({\mathcal F})-m$ have the same margins. A \emph{Markov basis}
${\mathcal M}$ is a finite set of Markov moves which makes
connected the set of all tables, or designs, with the same
margins. Notice that by adding Markov moves to a Markov basis
yields again a Markov basis.

In standard problems involving contingency tables, Theorem $3.1$
in \cite{diaconis|sturmfels:98} states that an actual way for
writing a Markov basis of a design matrix $A$ is as follows.
Compute a Gr\"obner basis of the toric ideal ${\mathcal I}_A$
(w.r.t. an arbitrary term-order) and then define the moves by
taking the logarithms of the binomials ($p^a-p^b\mapsto a-b$) in the Gr\"obner basis.

As observed in \cite{rapallo|rogantin:07} and
\cite{rapallo|yoshida:10}, when the cells of the table are
bounded, we need a special Markov basis, namely the Universal
Markov basis, derived from the Universal Gr\"obner basis and
taking logarithms. In our problem, the relevant tables are bounded
as each entry can be only $0$ or $1$. Therefore, in what follows
we consider the moves of the Universal Markov basis.

If we start from a fraction with matrix $N({\mathcal F}_0)$, the
Markov chain is then built as follows:
\begin{itemize}
\item at each step $i$, randomly choose a Markov move $m$ in
${\mathcal M}$ and a sign $\varepsilon \in \{ \pm 1 \}$;

\item if $N({\mathcal F}_i)+ \varepsilon m$ is a binary table,
move the chain to $N({\mathcal F}_{i+1}) = N({\mathcal F}_i)+
\varepsilon m$; otherwise, stay in $N({\mathcal F}_i)$.
\end{itemize}
The Markov chain described above is a connected chain over all the
designs with fixed margins, and its stationary distribution is the
uniform one. By considering the classical Metropolis-Hastings
probability ratio, one can define a Markov chain converging to any
specified probability distribution, see
\cite{diaconis|sturmfels:98}.

Now, a set of saturated fractions can be extracted from the Markov
chain above, by comparing each fraction with the supports of the relevant circuits, as described in Section
\ref{main-res}, and discarding the non-saturated fractions.

Two computational remarks are now in order: (a) the Universal
Gr\"obner basis coincide with the set of circuits, and the
Universal Markov basis does not need new computations; (b) due to
the limitation to $0-1$ tables, we can discard all the moves with
values outside $\{-1,0,1\}$ as they do not produce valid tables in
the algorithm above.

\begin{example}
We now reconsider the $2^4$ design with simple effects and 2-way
interactions, already illustrated in the previous sections.
Starting from the fraction ${\mathcal F}$ in Example
\ref{first-ex}, we are interested in saturated fractions with the
same one-way projections. The Universal Gr\"obner basis for this
problem coincides with the circuits basis and consists of $1,348$
elements, $532$ of them have values in $\{-1,0,1\}$ and are needed
to define the Markov chain. Taking such set of moves as input, we
are able to produce a sequence of fractions with the same one-way
projections as $\fraction_1$. Comparing each fraction with the
supports of the $120$ circuits found in Example \ref{first-ex}, we
have a random sample of saturated fractions. In less than $1$
minute, the execution of a simple {\tt R} function yields a sample
of $5,000$ saturated fractions with the desired projection.
\end{example}

\section{Conclusions} \label{fut-dir}

The theory described in this paper suggests several extensions and
applications. Firstly, it is interesting to explore how the results
can be extended for the characterization of saturated fractions to
more general designs.

Secondly, the connections between fractions and graphs need to be
studied. Indeed, it is known that the circuits for two-factor
designs can be derived by the complete bipartite graph associated
with the design, but little is known in the general case.

It would also be interesting to study the classification
of the saturated fractions with respect to some statistical
criteria. Among these criteria, we cite the minimum aberration in
a classical sense, or more recent tools, such as state polytopes.
Minimum aberration is a classical notion in this framework, and is
supported by a large amount of literature. This theory has been
developed in \cite{fries|hunter:80} and more recently in, e.g.,
\cite{ye:03} with the use of the indicator function in the
two-level case. The extension to the multilevel case is currently
an open problem. The use of state polytopes has been introduced in
\cite{bersteinetal:10}.

For applications, it would be interesting
to define an algorithm to optimize the search of saturated
fractions using the circuit basis in the construction of the
fractions instead of checking a random fraction.

Finally, the use of the inequivalent saturated fractions to
perform exact tests on model parameters is worth studying,
together with its implementation in statistical softwares, such as
SAS or {\tt R}. For inequivalent orthogonal arrays very
interesting results have already been achieved, see \cite{basso}
and \cite{arbfonrag}.

\bibliographystyle{spmpsci}
\bibliography{refsfrr}

{\bf Acknowledgment.} RF acknowledges SAS institute for providing software. FR is partially supported by the PRIN2009 grant number 2009H8WPX5.

\end{document}